\documentclass{ifacconf}

\usepackage{graphicx}      
\usepackage{natbib}        
\usepackage{amsfonts}
\usepackage{amsmath,amssymb}   
\usepackage{algorithm}
\usepackage{algorithmic}
\usepackage{enumitem}
\usepackage{subfigure}
\usepackage{threeparttable}
\usepackage{multirow}
\usepackage{booktabs}
\newtheorem{proof}{Proof}
\newtheorem{theorem}{Theorem}

\newtheorem{remark}{Remark}
\newtheorem{lemma}{Lemma}
\DeclareMathOperator*{\argmin}{arg\,min} 
\newcommand{\diag}{\mathop{\rm diag}\nolimits}

\newcommand{\rr}{{\mathbb R}}
\newcommand{\ba}[1]{\begin{array}{#1}}
\newcommand{\ea}{\end{array}}

\newcommand{\mr}[1]{\mathrm{#1}}
\begin{document}
\begin{frontmatter}

\title{Time-certified Input-constrained NMPC via Koopman Operator\thanksref{footnoteinfo}} 

\thanks[footnoteinfo]{This work was supported by the U.S. Food and Drug Administration under the FDA BAA-22-00123 program, Award Number 75F40122C00200. KG was also supported by the U.S. Department of Energy, Office of Science, Office of Advanced Scientific Computing Research, Department of Energy Computational Science Graduate Fellowship under Award Number DE-SC0022158.}

\author[First]{Liang Wu} 
\author[First]{Krystian Ganko} 
\author[First]{Ricahrd D. Braatz}

\address[First]{Massachusetts Institute of Technology, 
   Cambridge, MA 02139 USA (e-mail: \{liangwu,kkganko,braatz\}@mit.edu).}

\begin{abstract}
Determining solving-time certificates of nonlinear model predictive control (NMPC) implementations is a pressing requirement when deploying NMPC in production environments. Such a certificate guarantees that the NMPC controller returns a solution before the next sampling time. However, NMPC formulations produce nonlinear programs (NLPs) for which it is very difficult to derive their solving-time certificates. Our previous work, \cite{wu2023direct}, challenged this limitation with a proposed input-constrained MPC algorithm having exact iteration complexity but was restricted to linear MPC formulations. This work extends the algorithm to solve input-constrained NMPC problems, by using the Koopman operator and a condensing MPC technique. We illustrate the algorithm performance on a high-dimensional, nonlinear partial differential equation (PDE) control case study, in which we theoretically and numerically certify the solving time to be less than the sampling time.
\end{abstract}


\end{frontmatter}

\section{Introduction}
Model predictive control (MPC) is a model-based optimal control technique widely applied in a range of applications, including in manufacturing processes, energy systems, and robotics, see \cite{qin2003survey}. At each sampling time, MPC solves an on-line optimization which is formulated with a dynamical prediction model and user-specified constraints and objectives. 

Linear MPC is formulated using a linear process model, which leads to solving a quadratic program (QP). Nonlinear MPC (NMPC) instead adopts a nonlinear model, which results in a nonlinear program (NLP) that has a higher computational burden than the corresponding QP. Nearly all industrial systems are better described by nonlinear models, but deploying more computationally expensive NMPC under real-time process requirements for fast time-scale applications, such as robotics, is more challenging than deploying linear MPC.

A key requirement for deploying MPC in production environments is the execution speed (or, throughput), which is measured by both (\textit{i}) the average execution time (to free the processor to execute other tasks), and (\textit{ii}) the worst-case execution time which needs to be less than the sampling time. Most studies concentrate on developing algorithms with fast average execution time, see \cite{ferreau2014qpoases, stellato2020osqp, wu2023simple, wu2023construction}. The emphasis, however, should fall more on the certification of worst-case execution time, which is evident from the common assumption that the solution of the current MPC optimization task must be prepared before the arrival of the next sampling time, e.g., \cite[Assumption 5]{zavala2009advanced}. 

The worst-case execution time is computed from the worst-case number of floating-point operations (``FLOP") using the approximate relation
$$
\textrm{execution time} = \frac{\textrm{total \# FLOP in NLP solve}}{\textrm{average \# FLOP per second}},
$$
in which the denominator roughly depends on the embedded processor technology.\footnote{The number of flops used in each MPC calculation is sufficiently large that using an average in the denominator is highly accurate.} Determining the worst-case number of floating point operations subsequently requires ascertaining the worst-case number of iterations in the iterative optimization algorithm used by the MPC scheme. Certifying the number of iterations of an iterative optimization algorithm is particularly challenging for the on-line NMPC optimization. In most iterative optimization algorithms, the number of iterations depends on the data of the optimization such as the Hessian matrix and the gradient vector, and the data of the MPC optimization depends on the feedback states at each sampling time. More specifically for NMPC formulations, techniques that simplify the NLPs by reformulation as QPs (e.g., via successive online linearization or real-time iteration, see \cite{gros2020linear}) cause the Hessian matrix to also become time-varying. 

The remainder of this section summarizes past research in obtaining worst-case iteration numbers in linear MPC, and the novel extensions of this work to NMPC schemes. In \cite{giselsson2012execution, richter2011computational, bemporad2012simple}, accelerated gradient methods were used to solve the dual problem of linear MPC. The certification procedure on the worst-case number of iterations was also provided to determine the worst-case execution time. In these first-order methods, the iteration bound depends on the distance between the initial point and the optimal solution, which is unknown in advance and needs to be estimated. In practice, these methods are too conservative, i.e., their derived worst-case iteration bound is typically about two orders of magnitude larger than the actual number of iterations, see \cite{richter2011computational}.

In \cite{cimini2017exact, arnstrom2019exact, cimini2019complexity}, active set-based methods were used to solve QPs that arise from linear MPC, and the certification procedure of the iteration bound was described. These works rely on a technique that combines explicit MPC (i.e., off-line generation of a lookup table for the feedback control law) and implicit MPC (i.e., on-line optimization to solve for the feedback control law) to certify the iteration bound. Explicit MPC (\cite{bemporad2002explicit, alessio2009survey}) directly provides the certificate of the execution time as the worst-case search time in the lookup table. However, explicit MPC is practically limited to small and medium-sized problems where lookup table sizes are manageable. Another interesting work from \cite{okawa2021linear} also illustrated how to derive the iteration bound. The input-constrained MPC was first formulated as a linear complementarity problem (LCP) and then a \textit{modified N-step vector} was found via linear programming to ensure that the LCP is solved in $N$ iterations.

Unfortunately, the complexity of the procedures to derive certificates of worst-case execution time for these linear MPC methods hinders their extension to NMPC problems. Our previous work, \cite{wu2023direct}, first proposed an input-constrained linear MPC algorithm with the exact number of iterations,
$$
\mathcal{N}=\!\left\lceil\frac{\log(\frac{2n}{\epsilon})}{-2\log(\frac{\sqrt{2n}}{\sqrt{2n}+\sqrt{2}-1})}\right\rceil\! + 1.
$$
where $\left\lceil c\right\rceil$ maps $c$ to the least integer greater than or equal to $c$. This result is independent of the optimization problem data and dependent only on the number of variables $n$ and the stopping criterion $\epsilon$, making it suitable for parametric MPC problems. In this work, we extend the result to input-constrained nonlinear MPC problems by using the Koopman operator, which identifies a linear model by ``lifting" the state-space dimension of nonlinear dynamical systems. Our result is the first to produce a time-certified algorithm for input-constrained NMPC problems.

\section{Problem formulations}
Consider the input-constrained NMPC formulation
\begin{equation}\label{problem_input_NMPC}
    \begin{aligned}
        \text{(NMPC)}\quad\min\quad& J(\hat{x}_t)= l_N(x_N) + \sum_{k=0}^{N-1}l(x_k,u_k) \\
        \text{s.t. }\quad& x_0 = \hat{x}_t\\
        \quad& x_{k+1} = f(x_k,u_k), ~k\in\mathbb{Z}_{0,N-1},\\
        \quad& u_{\min} \leq u_k \leq u_{\max},~k\in\mathbb{Z}_{0,N-1},
    \end{aligned}    
\end{equation}
where $x_k\in\mathbb{R}^{n_x}, u_k\in\mathbb{R}^{n_u}$ denote the states and the control input, respectively, at time instance $k$ and $\hat{x}_t$ denotes the feedback states at the current sampling time $t$. The nonlinear function $f(\cdot):\mathbb{R}^{n_x}\times \in\mathbb{R}^{n_u}\rightarrow\mathbb{R}^{n_x}$ defines the dynamic model of the plant. The control inputs are constrained in $[u_{\min},u_{\max}]$, which are from the physical limits of the actuators. This article uses the standard formulation in which $l_N=\frac{1}{2}\|x_N-x_r\|_{W_N}^2$ and $l(x_k,u_k)=\frac{1}{2}\|u_k-u_r\|_{W_u}^2 + \frac{1}{2}\|x_k-u_r\|_{W_x}^2$, where $x_r,u_r$ are the targeted tracking references for the states and the control inputs, and $W_N, W_u, W_x$ denote the weight matrices for the terminal states, the control input, and the non-terminal states, respectively.

The formulation (NMPC) (\ref{problem_input_NMPC}) is an NLP that must be solved at every sampling time. There exist many efficient algorithms for solving (NMPC), but they lack the certificate of worst-case solving time---which, by the arguments above, is required for deploying NMPC in production environments. Here we first employ the Koopman operator to obtain a lifted high-dimensional linear predictor for the nonlinear dynamical system via data-driven models. Then, by condensing the Koopman-transformed MPC problem (i.e., eliminating the lifted high-dimensional states), the resulting problem becomes a box-constrained QP depending only on the control inputs.

\subsection{Koopman and Extended Dynamic Mode Decomposition}
\cite{koopman1931hamiltonian} and \cite{koopman1932dynamical} proposed an alternative perspective grounded in operator theory to represent the uncontrolled discrete-time nonlinear dynamical system $ x_{k+1}=f(x_k)$.
Koopman demonstrated the existence of an infinite-dimensional linear operator $\mathcal{K}$, which advances the evolution of an infinite-dimensional Hilbert space of measurement functions $\psi(x)$ described as
\begin{equation}
    \mathcal{K} \psi(x_k) \triangleq \psi(f(x_k)).
\end{equation}
Since Koopman operator theory was described first for autonomous dynamical systems, numerous schemes (see \cite{williams2016extending,proctor2018generalizing,korda2018linear}) have been proposed to extend the application of the Koopman operator to controlled systems of the form
\begin{equation}\label{eqn_nonlinear_system}
 x_{k+1}=f(x_k,u_k),
\end{equation}
where $u_k\in\rr^{n_u}$ denotes the control input of the system at time step $k$. To generalize the Koopman operator to (\ref{eqn_nonlinear_system}), we adopt the scheme from \cite{korda2018linear} which introduced an extended state vector as
\[
\mathcal{X}=\left[\begin{array}{@{}c@{}}
    x  \\
    \mathbf{u}
\end{array}\right]\!,
\]
where $\mathbf{u}\triangleq\left\{u_i\right\}_{i=0}^\infty\in l(\mathcal{U})$ and $u_i\in \mathcal{U}$ represents the control input sequence and $l(\mathcal{U})$ denotes the space of all control input sequences $\mathbf{u}$. The dynamics of the extended state $\mathcal{X}$ are described as 
\[
f_\mathcal{X}(\mathcal{X}) = \left[\begin{array}{@{}c@{}}
     f(x,\mathbf{u}(0))  \\
     \boldsymbol{S}\mathbf{u}
\end{array}\right]\!,
\]
where $\mathbf{u}(i)$ denotes the $i$th element of $\mathbf{u}$ and $\boldsymbol{S}$ represents the left shift operator, $(\boldsymbol{S}\mathbf{u})(i)\triangleq\mathbf{u}(i+1)$. Then the Koopman operator associated with the dynamics of the extended state can be defined on the set of extended observables $\phi(\mathcal{X})$ as
\[
\mathcal{K}\phi(\mathcal{X})\triangleq\phi(f_\mathcal{X}(\mathcal{X})).
\]
The infinite-dimensional Koopman operator must be truncated in practice, and several finite-dimensional approximations have been proposed (see, e.g., \cite{williams2015data,williams2016extending,korda2018linear}) which employ a data-driven Extended Dynamic Mode Decomposition (EDMD) algorithm. In EDMD specifically, the set of extended observables is designed as the ``lifted" mapping
\begin{equation}
    \phi(x,\mathbf{u}) = \left[\begin{array}{@{}c@{}}
        \psi(x) \\
        \mathbf{u}(0)
    \end{array}\right]\!,
\end{equation}
where $\psi(x)\triangleq\left[\psi_1(x), \cdots{}, \psi_{n_\psi}(x)\right]^\top$, $n_\psi$ is the designed number of observables (with $n_\psi \gg n_x$), and $\mathbf{u}(0)$ denotes the first component of the sequence $\mathbf{u}$.

The EDMD approach expands the nonlinear observables $\phi(x,\mathbf{u})$ in a basis function set, e.g., Radial Basis Functions used in \cite{korda2018linear}, instead of directly solving for them via optimization. Only the Koopman operator is learned via an optimization procedure. In particular, the approximate Koopman operator identification problem is reduced to a least-squares problem, which assumes that the sampled data $\{(x_j,\mathbf{u}_j),(x_j^+,\mathbf{u}_j^+)\}~\forall j=1,\cdots, N_d$ are collected with the update mapping
\[
\left[\begin{array}{@{}c@{}}
    x_j^+ \\
    \mathbf{u}_j^+
\end{array}\right]\! =\! \left[\begin{array}{@{}c@{}}
    f(x_j,\mathbf{u}_j(0))  \\
     \boldsymbol{S}\mathbf{u}_j
\end{array}\right],
\]
where the superscript $+$ denotes the value at the next time step. Then an approximation of the Koopman operator, $\mathcal{A}$, is obtained by solving
\begin{equation}\label{problem_EDMD_original}
J(\mathcal{A}) = \min_{\mathcal{A}}\sum_{j=1}^{N_d}\|\phi(x_j^+,\mathbf{u}_j^+)-\mathcal{A}\phi(x_j,\mathbf{u}_j)\|^2. 
\end{equation}
Since there is no need to predict the future control input sequence, the last $n_u$ rows of $\mathcal{A}$ can be discarded. Additionally, let $\bar{\mathcal{A}}$ be the remaining part of $\mathcal{A}$ after discarding the part associated with the future control input. Then $\bar{\mathcal{A}}$ can be decomposed into $A\in\rr^{n_\psi\times n_\psi}$ and $B\in\rr^{n_\psi\times n_u}$ as 
\[
\bar{\mathcal{A}} = \left[A, B\right]
\]
so that the problem (\ref{problem_EDMD_original}) can be reduced to
\begin{equation}\label{problem_EDMD_reduced}
    J(A,B) = \min_{A,B} \sum_{j=1}^{N_d}\|\psi(x_j^+)-A\psi(x_j)-B\mathbf{u}_j(0)\|^2.
\end{equation}
We finally obtain the identified linear predictor model in the ``lifted" space as
\begin{equation}
    \psi_{k+1} = A \psi_k + Bu_k,
\end{equation}
where $\psi_k\triangleq\psi(x_k) \in\rr^{n_\psi}$ denotes the lifted state space. Additionally, the output matrix $C$ is obtained as the best projection of $x$ onto the span of $\psi$ in a least-squares sense, i.e., as the solution to
\begin{equation}\label{problem_EDMD_C}
    J(C) = \min_C \sum_{j=1}^{N_d}\|x_j-C\psi(x_j)\|^2.
\end{equation}
At the end, a linear model for $y$ can be formulated using
\[
y_k = C \psi_k.
\]
\begin{remark}
As \cite{korda2018linear} claims, if the designed lifted mapping $\psi(x)$ contains the state $x$ after the re-ordering 
$\psi(x)= [x^\top, \hat{\psi}(x)]^\top$, then the solution to (\ref{problem_EDMD_C}) is $C=[I,0]$.
\end{remark}
\subsection{Transforming NMPC to condensed MPC}
After obtaining the approximate lifted predictor, NMPC can now be transformed into the MPC problem:
\begin{equation}\label{eqn_KMPC}
    \begin{aligned}
        \min\quad&J(\hat{x}_t)=\tfrac{1}{2}\|C\psi_{N}-x_r\|_{W_N}^2\\
        &\quad+\tfrac{1}{2}\!\sum_{k=0}^{N_-1}\|u_k-u_r\|_{W_u}^2+\|C\psi_k-x_r\|_{W_x}^2\\
        \text{s.t.}\quad&\psi_0 = \psi(\hat{x}_t),\\
        \quad& \psi_{k+1}=A\psi_k+Bu_k,~k\in\mathbb{Z}_{0,N-1},\\
        \quad& -e \leq u_k \leq e,~k\in\mathbb{Z}_{0,N-1},
    \end{aligned}
\end{equation}
where the control inputs are assumed to have been scaled into the unit box constraints $[-e,e]$.

The main drawback of the Koopman operator is that the extremely high-dimensional lifted state space vector may increase the computational cost. This potential concern can be avoided by using the condensed MPC problem formulation. Define $z\triangleq \operatorname{col}(u_0,\cdots{},u_{N-1})\in\mathbb{R}^{n}$, where $n=N\times n_u$,  $\bar{Q}\triangleq\diag(C^\top W_xC,\cdots{},C^\top W_xC,C^\top W_NC)$, $\bar{R}\triangleq\diag(W_u,\cdots{},W_u)$,
\[
S\triangleq\!\left[\begin{array}{cccc}
      B &  0 & \cdots & 0\\
      AB & B & \cdots & 0\\
      \vdots & \vdots & \ddots & \vdots\\
      A^{N-1}B & A^{N-2}B & \cdots & B
 \end{array}\right]\!\!, \ H\triangleq\Bar{R}+S^\top\Bar{Q}S.
\]
These matrices are calculated off-line, so their computation cost is not included in the on-line computational cost. Then (\ref{eqn_KMPC}) is equivalently constructed as
\begin{subequations}\label{problem_Box_QP}
    \begin{align}
        z^*=&\argmin_z J(\hat{x}_t)= \tfrac{1}{2}z^\top H z + z^\top h\label{problem_Box_QP_objective}\\
        &\textrm{s.t. } -e \leq z \leq e,\label{proble_Box_QP_box_constraint}
    \end{align}
\end{subequations}
where
\begin{equation}\label{eqn_h_g}
h\triangleq S^\top\bar{Q}g - \! \left[\begin{array}{@{}c@{}}
     W_uu_r  \\
     \vdots \\
     W_uu_r\\
     W_uu_r
\end{array}\right]\!\!,
\ g\triangleq\!\left[\begin{array}{@{}c@{}}
    A \\
    \vdots \\
    A^{N-2} \\
    A^{N-1}
 \end{array}\right]\!\psi_0  - \left[\begin{array}{c}
       C^\top W_x x_r  \\
      \vdots \\
       C^\top W_x x_r \\
        C^\top W_N x_r
 \end{array}\right]  
\end{equation}
needs to be computed on-line since $\psi_0=\psi(\hat{x}_t)$ and $x_r,u_r$ are time-varying.

\section{Time-certified IPM algorithm}
We solve the Box-QP (\ref{problem_Box_QP}) by adopting the path-following full-Newton Interior Point Method (IPM) with the exact number of iterations from our previous work \cite{wu2023direct}. Its Karush–Kuhn–Tucker (KKT) conditions are
\begin{subequations}\label{eqn_KKT}
\begin{align}
    Hz + h + \gamma - \theta = 0\label{eqn_KKT_a},\\
    z + \alpha - e=0\label{eqn_KKT_b},\\
    z - \omega + e=0\label{eqn_KKT_c},\\
    \gamma \alpha = 0\label{eqn_KKT_d},\\
    \theta \omega = 0\label{eqn_KKT_e},\\
    (\gamma,\theta,\alpha,\omega)\geq0.
\end{align}
\end{subequations}
The path-following IPM introduces a positive parameter $\tau$ to replace (\ref{eqn_KKT_d}) and (\ref{eqn_KKT_e}) by
\begin{subequations}\label{eqn_KKT_tau}
\begin{align}
    \gamma \alpha = \tau^2 e\label{eqn_KKT_tau_d},\\
    \theta \omega = \tau^2 e\label{eqn_KKT_tau_e}.
\end{align}
\end{subequations}
It is well known that, as $\tau$ approaches 0, the path $(z_{\tau},\gamma_{\tau},\theta_{\tau},\alpha_{\tau},\omega_{\tau})$ approaches a solution of \eqref{eqn_KKT}. 

The feasible path-following IPM algorithm has the best theoretical iteration complexity of  $O(\sqrt{n})$. In addition, our algorithm is based on feasible IPM wherein all iterates lie in the strictly feasible set
\[
\mathcal{F}^0\triangleq\{(z,\gamma,\theta,\alpha,\omega)\lvert~\eqref{eqn_KKT_a}\textrm{--}\eqref{eqn_KKT_c}\text{ satisfied},(\gamma,\theta,\alpha,\omega)>0\}.
\]

\subsection{Strictly feasible initial point}
Our previous \cite{wu2023direct} proposed a novel cost-free initialization strategy to find a strictly feasible initial point that also satisfies the specific conditions. First, an obvious strictly feasible initial point is
\[
z^0 = 0,\quad \gamma^0= \|h\|_\infty - \tfrac{1}{2}h,\quad \theta^0 =\|h\|_\infty + \tfrac{1}{2}h, 
\]
\[\alpha^0 = e, \quad\omega^0 = e,
\]
where $\|h\|_\infty=\max \{ |h_1|,|h_2|,\cdots,|h_{n}|\}$. It is straightforward to see that the above initial point strictly lies in $\mathcal{F}^0$.

\begin{remark}[Initialization strategy]\label{remark_initialization_strategy}
If $h=0$, the optimal solution of problem (\ref{problem_Box_QP}) is $z^*=0$; in the case of $h\neq0$, we first scale the objective (\ref{problem_Box_QP_objective}) (which does not change the optimal solution) as
\[
\min_z \tfrac{1}{2} z^\top \!\left(\frac{2\lambda}{\|h\|_\infty}H\right) \!z + z^\top \!\left(\frac{2\lambda}{\|h\|_\infty}h\right)\!.
\]
With the definitions $\tilde{H} = \frac{H}{\|h\|_\infty}$ and $\tilde{h}=\frac{h}{\|h\|_\infty}$, $\|\tilde{h}\|_\infty=1$ and (\ref{eqn_KKT_a}) can be replaced by
\[
2\lambda \tilde{H}z+2\lambda\tilde{h}+\gamma-\theta=0,
\]
and the initial points
\begin{equation}\label{eqn_initialization_stragegy}
    z^0 = 0,\ \gamma^0=1 - \lambda \tilde{h},\ \theta^0 =1 + \lambda \tilde{h},\ \alpha^0 = e, \ \omega^0 = e,
\end{equation}
\end{remark}
can be adopted, where
\[
\lambda =\frac{1}{\sqrt{n+1}}.
\]
It is straightforward to verify that (\ref{eqn_initialization_stragegy}) lies in $\mathcal{F}^0$. The reason to use the scale factor $\frac{2\lambda}{\|h\|_\infty}$ is to make the initial point satisfy the neighborhood requirements, e.g., see \cite[Lemma 4]{wu2023direct}. 

\subsection{Newton direction}
Denote $v=\operatorname{col}(\gamma, \theta) \in \mathbb{R}^{2n}$ and $s=\operatorname{col}(\alpha,\omega) \in \mathbb{R}^{2n}$. Then replace (\ref{eqn_KKT_tau_d}) and (\ref{eqn_KKT_tau_e}) by $v s = \tau^2e$ to obtain the new complementary condition
\begin{equation}
    \sqrt{vs} = \sqrt{\tau^2e}\label{eqn_new_complementary}.
\end{equation}
From \textit{Remark \ref{remark_initialization_strategy}}, $(z, v, s)\in \mathcal{F}^0$ and a direction $(\Delta z,\Delta v,\Delta s)$ can be obtained by solving the system of linear equations
\begin{subequations}\label{eqn_newKKT_compact}
    \begin{align}
        &2\lambda\tilde{H}\Delta z + \Omega \Delta v = 0\label{eqn_newKKT_compact_a},\\
        & \Omega^\top \Delta z + \Delta s = 0\label{eqn_newKKT_compact_b},\\
        &\sqrt{\frac{s}{v}}\Delta v + \sqrt{\frac{v}{s}}\Delta s = 2(\tau e-\sqrt{v s}),\label{eqn_newKKT_compact_c}
    \end{align}
\end{subequations}
where $\Omega=[I,-I] \in\mathbb{R}^{n \times 2n}$. Letting
\begin{subequations}\label{eqn_Delta_gamma_theta_phi_psi}
    \begin{align}
        &\Delta \gamma=\frac{\gamma}{\alpha}\Delta z+2\!\left(\sqrt{\frac{\gamma}{\alpha}}\tau e-\gamma\right)\!,\\
        &\Delta \theta=-\frac{\theta}{\omega}\Delta z+2\!\left(\sqrt{\frac{\theta}{\omega}}\tau e-\theta\right)\!,\\
        &\Delta\alpha = - \Delta z,\\
        &\Delta\omega = \Delta z
    \end{align}
\end{subequations}
reduces (\ref{eqn_newKKT_compact}) into a more compact system of linear equations,
\begin{equation}{\label{eqn_compact_linsys}}
    \begin{aligned}
        &\left(2\lambda\tilde{H}+\diag\!\left(\frac{\gamma}{\alpha}\right)\!+ \diag\!\left(\frac{\theta}{\omega}\!\right)\! \right) \!\Delta z\\
        &=2\!\left(\sqrt{\frac{\theta}{\omega}}\tau e-\sqrt{\frac{\gamma}{\alpha}}\tau e+ \gamma - \theta\right)
    \end{aligned}
\end{equation}

\subsection{Iteration complexity and algorithm implementation}
Let's denote $\beta\triangleq \sqrt{vs}$ and define the proximity measure
\begin{equation}\label{eqn_xi}
\xi(\beta,\tau)=\frac{\|\tau e-\beta\|}{\tau}.
\end{equation}
\begin{lemma}[See \cite{wu2023direct}]\label{lemma_strictly_feasible}
Let $\xi:=\xi(\beta,\tau) < 1$. Then the full Newton step is strictly feasible, i.e., $v_{+}>0$ and $s_{+}>0$.
\end{lemma}
\begin{lemma}[See \cite{wu2023direct}]\label{lemma_duality_gap}
    After a full Newton step, let $v_{+}=v+\Delta v$ and $s_{+}=s+\Delta s$, then the duality gap is
    \[
        v_{+}^Ts_{+}\leq(2n)\tau^2.
    \]
\end{lemma}
\begin{lemma}[See \cite{wu2023direct}]\label{lemma_xi}
Suppose that $\xi=\xi(\beta,\tau)<1$ and $\tau_+=(1-\eta)\tau$ where $0<\eta<1$. Then
\[
\xi_+=\xi(\beta_+,\tau_+)\leq\frac{\xi^2}{1+\sqrt{1-\xi^2}}+\frac{\eta\sqrt{2n}}{1-\eta}.
\]
Furthermore, if $\xi\leq\frac{1}{\sqrt{2}}$ and $\eta=\frac{\sqrt{2}-1}{\sqrt{2n}+\sqrt{2}-1}$, then $\xi_+\leq\frac{1}{\sqrt{2}}$.
\end{lemma}
\begin{lemma}[See \cite{wu2023direct}]\label{lemma_xi_condition}
The value of $\xi(\beta,\tau)$ before the first iteration is denoted as
$\xi^0=\xi(\beta^0,(1-\eta)\tau^0)$. If $(1-\eta)\tau^0=1$ and $\lambda=\frac{1}{\sqrt{n+1}}$, then $\xi^0\leq\frac{1}{\sqrt{2}}$ and $\xi(\beta, w)\leq\frac{1}{\sqrt{2}}$ are always satisfied.
\end{lemma}
\begin{lemma}[See \cite{wu2023direct}]\label{lemma_exact}
    Let $\eta=\frac{\sqrt{2}-1}{\sqrt{2n}+\sqrt{2}-1}$ and $\tau^0=\frac{1}{1-\eta}$, Algorithm \ref{alg_time_certifed_IPM} exactly requires    \begin{equation}
\mathcal{N}=\left\lceil\frac{\log(\frac{2n}{\epsilon})}{-2\log(\frac{\sqrt{2n}}{\sqrt{2n}+\sqrt{2}-1})}\right\rceil\! + 1
    \end{equation}  
    iterations, the resulting vectors being $v^\top s\leq\epsilon$. 
\end{lemma}

\begin{theorem}\label{theorem}
Let $m_\textrm{lifting}$ denote the number of FLOP required by the lifting mapping. Then Algorithm \ref{alg_time_certifed_IPM} requires at most $m_\textrm{lifting} + (2Nn_\psi^2 + \frac{N(N+1)n_un_\psi}{2} + Nn_xn_\psi+Nn_u^2+ 2n) + n + 5n+3 + \mathcal{N}\!\left( 1 + \frac{1}{3}n^3+\frac{1}{2}n^2\right.$ $\left.+\frac{1}{6}n  + 2n^2 +  10n + 5n\right)$ FLOP.
\end{theorem}
\begin{proof}
    In Algorithm \ref{alg_time_certifed_IPM}: Step 1 takes $m_\textrm{lifting}$ and $(2Nn_\psi^2 + \frac{N(N+1)n_un_\psi}{2} + Nn_xn_\psi+Nn_u^2+ 2n)$ FLOP, which can be achieved by an efficient implementation of (\ref{eqn_h_g}); Step 2 takes $n$ FLOP to find the infinity norm of $h$; Step 3 takes $5n+3$ FLOP to assign the values for 5 vectors and 3 scalars. Each iteration of Step 4 takes in total $\left( 1 + \frac{1}{3}n^3+\frac{1}{2}n^2\right.$ $\left.+\frac{1}{6}n + 2n^2 + 10n + 5n\right)$ FLOP.
\end{proof}
\begin{remark}
    By Theorem \ref{alg_time_certifed_IPM}, the lifted high-dimensional states brought by the Koopman operator only slightly increase the on-line computation cost.
\end{remark} 

\begin{algorithm}[H]
    \caption{A time-certified IPM algorithm for input-constrained Koopman MPC (\ref{eqn_KMPC})
    }\label{alg_time_certifed_IPM}
    \textbf{Input}: the current feedback states $\hat{x}_t$, the state reference signal $x_r$, the lifting mapping $\psi$, and the stopping tolerance $\epsilon$; the required exact number of iterations $\mathcal{N}=\left\lceil\frac{\log(\frac{2n}{\epsilon})}{-2\log(\frac{\sqrt{2n}}{\sqrt{2n}+\sqrt{2}-1})}\right\rceil\!+1$.
    \vspace*{.1cm}\hrule\vspace*{.1cm}
    \begin{enumerate}[label*=\arabic*., ref=\theenumi{}]
        \item $\psi_0\leftarrow\psi(\hat{x}_t)$ and calculate $h$ from (\ref{eqn_h_g});\\
        \item \textbf{if }$\|h\|_\infty=0$, \textbf{return} $z^*=0$; \textbf{otherwise},\\
        \item  $(z,\gamma,\theta,\phi,\psi)$ are initialized from (\ref{eqn_initialization_stragegy}) where $\lambda\leftarrow\frac{1}{\sqrt{n+1}}$, $\eta\leftarrow\frac{\sqrt{2}-1}{\sqrt{2n}+\sqrt{2}-1}$ and $\tau\leftarrow\frac{1}{1-\eta}$;\\
        \item \textbf{for} $k=1, 2,\cdots{}, \mathcal{N}$ \textbf{do}
        \begin{enumerate}[label=\theenumi{}.\arabic*., ref=\theenumi{}.\arabic*]
        \item $\tau\leftarrow(1-\eta)\tau$;
        \item solve \eqref{eqn_compact_linsys} for $\Delta z$ by using the Cholesky decomposition method with one forward substitution and one backward substitution;
        \item calculate $(\Delta\gamma,\Delta\theta,\Delta\alpha,\Delta\omega)$ from \eqref{eqn_Delta_gamma_theta_phi_psi};
        \item  $z\leftarrow z+\Delta z$, $\gamma\leftarrow \gamma+\Delta \gamma$, $\theta\leftarrow \theta+\Delta \theta$, $\alpha\leftarrow \alpha+\Delta \alpha$, $\omega\leftarrow \omega+\Delta \omega$;
        \end{enumerate}
        \item \textbf{end}
    \end{enumerate}
\end{algorithm}

\section{Nonlinear PDE control case study}
This section illustrates the effectiveness of our time-certified algorithm for a  nonlinear PDE control example. The PDE plant under consideration is the nonlinear Korteweg-de Vries (KdV) equation that models the propagation of acoustic waves in plasma or shallow water waves (see \cite{miura1976korteweg}) as
\begin{equation}
\frac{\partial y(t, x)}{\partial t}+y(t, x) \frac{\partial y(t, x)}{\partial x}+\frac{\partial^3 y(t, x)}{\partial x^3}=u(t,x)
\end{equation}
where $x\in[-\pi,\pi]$ is the spatial variable. We consider the control input $u$ to be $u(t,x)=\sum_{i=1}^4u_i(t)v_i(x)$, in which the four coefficients $\{u_i(t)\}$ are subject to the constraint $[-1,1]$ and are computed by the model predictive controller, and $v_i(x)$ are predetermined spatial profiles given as $v_i(x)=e^{-25(x-m_i)^2}$, with $m_1=-\pi/2$, $m_2=-\pi/6$, $m_3=\pi/6$, and  $m_4=\pi/2$.

The control objective is for the spatial profile $y(t,x)$ to track the given reference signal. In our closed-loop simulation, we discretize the $x$-axis of the nonlinear KdV equation at $N=128$ nodes, and adopt a spectral method involving the Fourier transform and split stepping to solve the nonlinear KdV equation, e.g., see \cite{KdV2012website}. The sampling time is chosen as $\Delta t=0.01$ s for data generation and the model predictive controller. The setting for our closed-loop simulation includes:
\begin{itemize}
    \item[\textit{i)}] \textit{Data generation:} The data are collected from 1000 simulation trajectories with 200 samples. At each simulation, the initial condition of the spatial profile is a random combination of four given spatial profiles, i.e., $y_1^0(0,x)=e^{-(x-\pi/2)^2}$, $y_2^0(0,x)=-\operatorname{sin}(x/2)^2$, $y_3^0(0,x)=e^{-(x+\pi/2)^2}$, $y_4^0(0,x)=\operatorname{cos}(x/2)^2$. The four control inputs $u_i(t)$ are distributed uniformly in $[-1,1]$.
    \item[\textit{ii)}] \textit{Koopman predictor:} We choose the lift function $\psi$ consisting of the origin states (128 spatial nodes), the constant $1$, the elementwise product of the origin states with one element shift, and the elementwise square of the origin states, which leads to the lifted state dimension $N_{\textit{lift}}=3\times 128+1=385$. Then the lifted linear predictor with $A\in\mathbb{R}^{385\times 385}$ and $B\in\mathbb{R}^{385\times 4}$ is obtained from the Moore-Penrose pseudoinverse of the lifting data matrix, and its output matrix is $C=[I_{128},0]\in\mathbb{R}^{128\times 385}$.
    \item[\textit{iii)}] \textit{MPC settings:} Set the prediction horizon $N=10$, the state cost matrix $W_x=W_N=I_{128}$, and the control inputs matrix $W_u=0.01 I_{4}$, and the control inputs are subject to $[-1,1]$. The state references $x_r\in\mathbb{R}^{128}$ are piecewise constant taking the values $[0.5,0.25,0,0.75]\times \text{ones}(128,1)$ for a 50 s simulation time, with the control input reference $u_r=0$.
\end{itemize}

Before the closed-loop simulation, we can calculate the worst-case total floating operations required at each sampling time. The dimension of the resulting Box-QP problem (\ref{problem_Box_QP}) is $n=4\times N=40$, and we adopt the stopping criteria $\epsilon=1\times10^{-6}$, so the required number of iterations is 
\[
\mathcal{N}=\left\lceil\frac{\log(\frac{2\times 40}{1e-6})}{-2\log(\frac{\sqrt{2\times 40}}{\sqrt{2\times 40}+\sqrt{2}-1})}\right\rceil\! + 1= 202.
\]
Further, the FLOP for the lifting mapping is $m_{\textrm{lifting}}=2N=256$, and, by Theorem \ref{theorem}, Algorithm \ref{alg_time_certifed_IPM} exactly require $256+(2\times10\times(385)^2+5\times11\times4\times385+10\times128\times385+10\times4\times385+2\times40)+40+200+3+202\times(1+1/3(40)^3+1/2(40)^2+40/6+2(40)^2+600) = 0.0088\times 10^{9}$ FLOP, which approximately leads to the execution time $0.0066$ s on a machine with 1 GFLOP/s computing power (a trivial requirement for most processors today). Thus, we can get a certificate that the execution time will be less than the adopted sampling time $\Delta t=0.01$ s.

In context, we ran our closed-loop simulation on a modern MacBook Pro with 2.7~GHz 4-core Intel Core i7 processors and 16GB RAM. Algorithm~\ref{alg_time_certifed_IPM} is executed in MATLAB2023a via a C-mex interface. The number of iterations was exactly $202$, and the maximum execution time was about $0.0075$ s less than $\Delta t=0.01$s. The closed-loop simulation results are plotted in Fig.\  \ref{fig_kdv}, which shows that the MPC algorithm provides quick and accurate tracking of the spatial profile $y(t,x)$ to the given reference profile. The control inputs do not violate $[-1,1]$, and also track the control input reference $u_r=0$ well. 
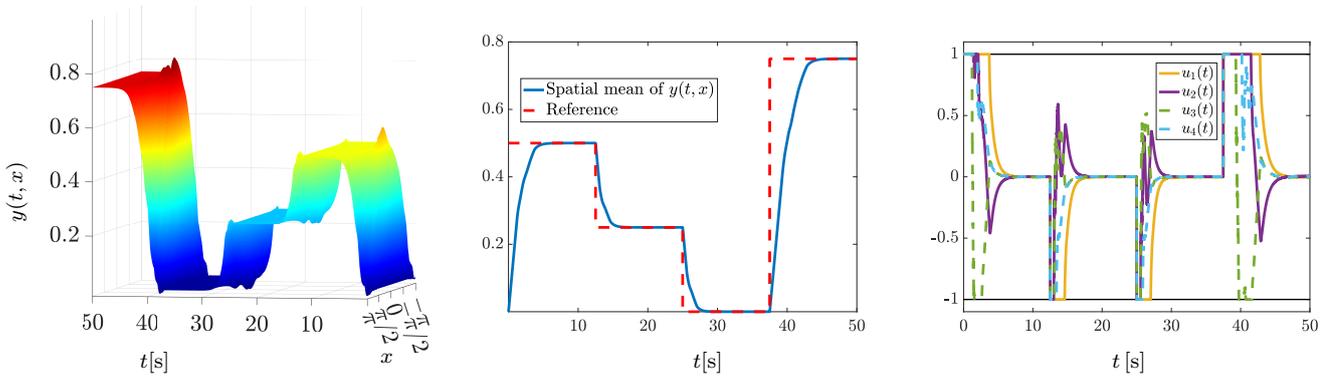
\begin{figure*}[!t]\label{fig}
\begin{picture}(140,110)
\put(-30,-15){\includegraphics[width=75mm]{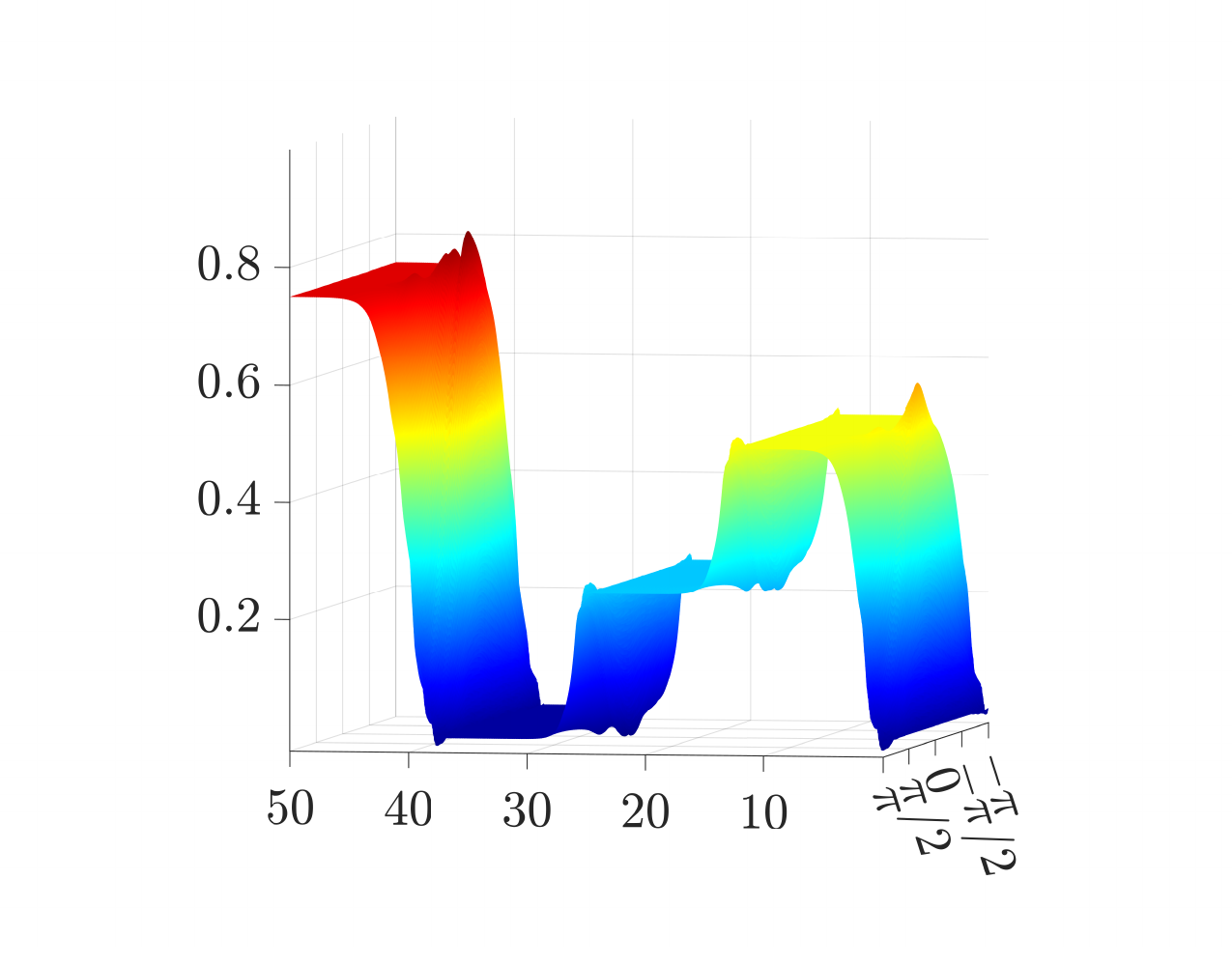}}
\put(160,0){\includegraphics[width=55mm]{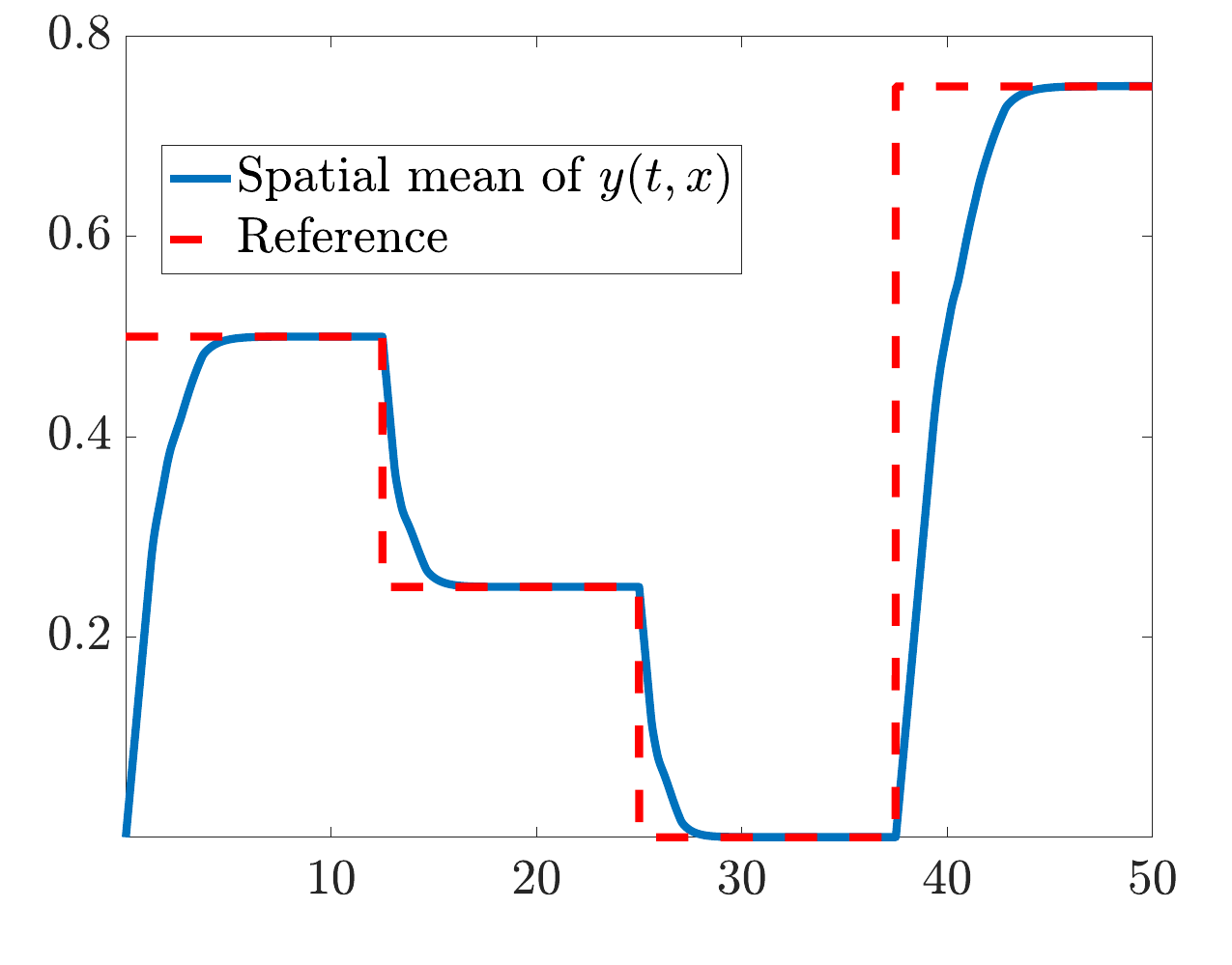}}
\put(330,0){\includegraphics[width=55mm]{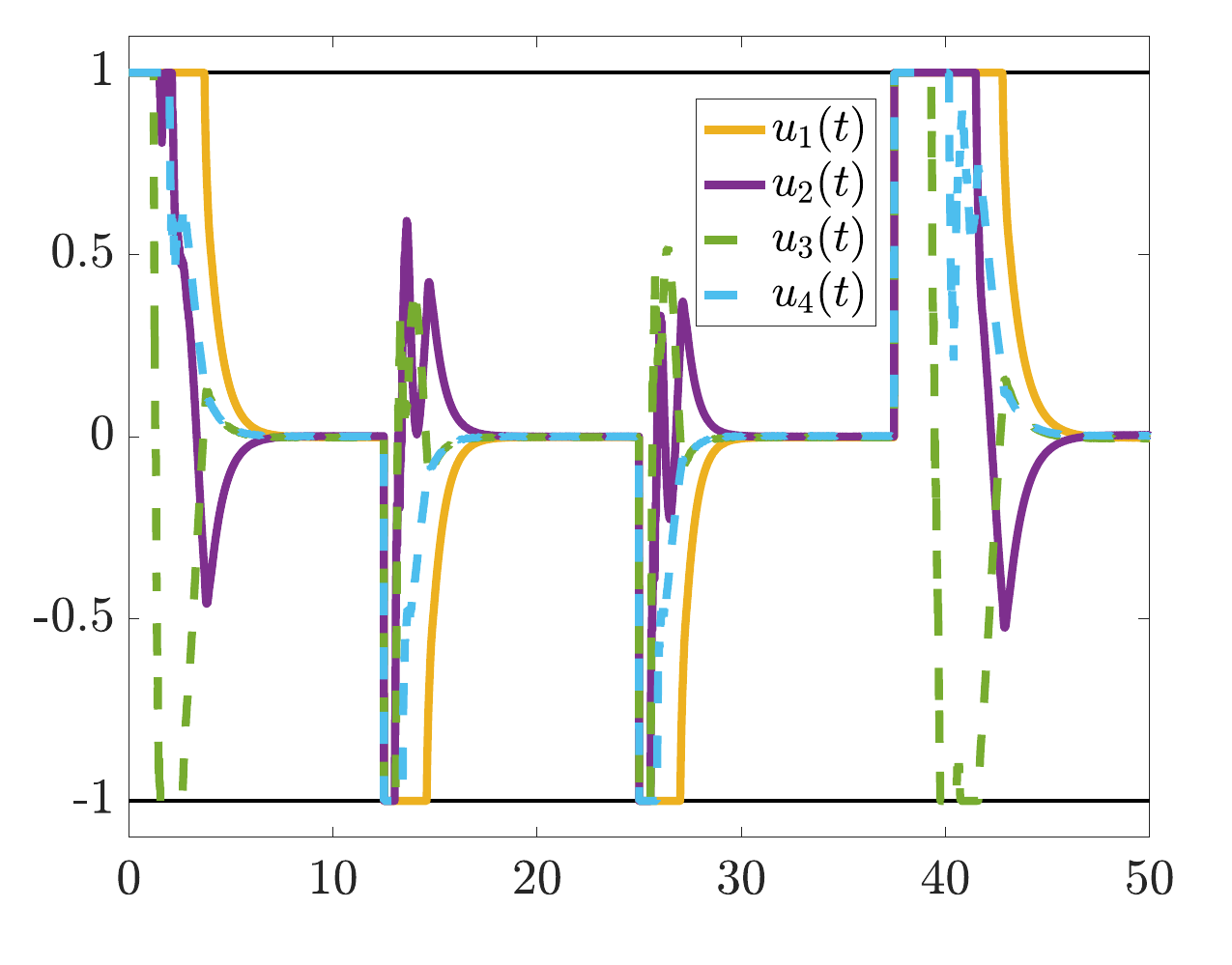}}
\put(128,-3){\footnotesize $x$}
\put(38,-5){\footnotesize $t [\mr{s}]$}
\put(-12,50){\footnotesize \rotatebox{90}{$y(t,x)$}}
\put(245,-5){\footnotesize $t [\mr{s}]$}
\put(402,-5){\footnotesize $t\,[\mr{s}]$}
\end{picture}
\caption{Closed-loop simulation of the nonlinear KdV system with MPC controller -- Tracking a piecewise constant spatial profile reference. Left: time evolution of the spatial profile $y(t,x)$. Middle: spatial mean of the $y(t,x)$. Right: the four control inputs.}
\label{fig_kdv}
\end{figure*}

\section{Conclusion}
This article proposes a time-certified algorithm for input-constrained NMPC problems, in which Koopman operator is used to identify a lifted high-dimensional linear predictor model. The resulting small Box-QP is formulated by eliminating the lifted high-dimensional states, and finally, our previous time-certified algorithm (see \cite{wu2023direct}) is applied to solve the Box-QP. Future work includes (\textit{i}) improving the computation efficiency further while preserving the time-certified feature; and (\textit{ii}) extension to general NMPC problems while preserving the time-certified feature.

\bibliography{ifacconf}             

\begin{thebibliography}{25}
\providecommand{\natexlab}[1]{#1}
\providecommand{\url}[1]{\texttt{#1}}
\providecommand{\urlprefix}{URL }
\expandafter\ifx\csname urlstyle\endcsname\relax
  \providecommand{\doi}[1]{doi:\discretionary{}{}{}#1}\else
  \providecommand{\doi}{doi:\discretionary{}{}{}\begingroup \urlstyle{rm}\Url}\fi

\bibitem[{Alessio and Bemporad(2009)}]{alessio2009survey}
Alessio, A. and Bemporad, A. (2009).
\newblock A survey on explicit model predictive control.
\newblock In L.~Magni, D.M. Raimondo, and F.~Allg\"{o}wer (eds.), \emph{Nonlinear Model Predictive Control: Towards New Challenging Applications}, 345--369. Springer, Berlin--Heidelberg.

\bibitem[{Arnstr{\"o}m and Axehill(2019)}]{arnstrom2019exact}
Arnstr{\"o}m, D. and Axehill, D. (2019).
\newblock Exact complexity certification of a standard primal active-set method for quadratic programming.
\newblock In \emph{Proceedings of the 58th IEEE Conference on Decision and Control}, 4317--4324.

\bibitem[{Bemporad et~al.(2002)Bemporad, Morari, Dua, and Pistikopoulos}]{bemporad2002explicit}
Bemporad, A., Morari, M., Dua, V., and Pistikopoulos, E.N. (2002).
\newblock The explicit linear quadratic regulator for constrained systems.
\newblock \emph{Automatica}, 38(1), 3--20.

\bibitem[{Bemporad and Patrinos(2012)}]{bemporad2012simple}
Bemporad, A. and Patrinos, P. (2012).
\newblock Simple and certifiable quadratic programming algorithms for embedded linear model predictive control.
\newblock \emph{IFAC Proceedings Volumes}, 45(17), 14--20.

\bibitem[{Cimini and Bemporad(2017)}]{cimini2017exact}
Cimini, G. and Bemporad, A. (2017).
\newblock Exact complexity certification of active-set methods for quadratic programming.
\newblock \emph{IEEE Transactions on Automatic Control}, 62(12), 6094--6109.

\bibitem[{Cimini and Bemporad(2019)}]{cimini2019complexity}
Cimini, G. and Bemporad, A. (2019).
\newblock Complexity and convergence certification of a block principal pivoting method for box-constrained quadratic programs.
\newblock \emph{Automatica}, 100, 29--37.

\bibitem[{Ferreau et~al.(2014)Ferreau, Kirches, Potschka, Bock, and Diehl}]{ferreau2014qpoases}
Ferreau, H., Kirches, C., Potschka, A., Bock, H., and Diehl, M. (2014).
\newblock qp{OASES}: {A} parametric active-set algorithm for quadratic programming.
\newblock \emph{Mathematical Programming Computation}, 6, 327--363.

\bibitem[{Giselsson(2012)}]{giselsson2012execution}
Giselsson, P. (2012).
\newblock Execution time certification for gradient-based optimization in model predictive control.
\newblock In \emph{Proceedings of the 51st IEEE Conference on Decision and Control}, 3165--3170.

\bibitem[{Gros et~al.(2020)Gros, Zanon, Quirynen, Bemporad, and Diehl}]{gros2020linear}
Gros, S., Zanon, M., Quirynen, R., Bemporad, A., and Diehl, M. (2020).
\newblock From linear to nonlinear {MPC}: {Bridging} the gap via the real-time iteration.
\newblock \emph{International Journal of Control}, 93(1), 62--80.

\bibitem[{Koopman(1931)}]{koopman1931hamiltonian}
Koopman, B. (1931).
\newblock Hamiltonian systems and transformation in {H}ilbert space.
\newblock \emph{Proceedings of the National Academy of Sciences}, 17(5), 315--318.

\bibitem[{Koopman and Neumann(1932)}]{koopman1932dynamical}
Koopman, B. and Neumann, J. (1932).
\newblock Dynamical systems of continuous spectra.
\newblock \emph{Proceedings of the National Academy of Sciences}, 18(3), 255--263.

\bibitem[{Korda and Mezi{\'c}(2018)}]{korda2018linear}
Korda, M. and Mezi{\'c}, I. (2018).
\newblock Linear predictors for nonlinear dynamical systems: {K}oopman operator meets model predictive control.
\newblock \emph{Automatica}, 93, 149--160.

\bibitem[{Meylan(2012)}]{KdV2012website}
Meylan, M. (2012).
\newblock Numerical solution of the {KdV}.
\newblock \urlprefix\url{wikiwaves.org/Numerical\_Solution\_of\_the\_KdV}.

\bibitem[{Miura(1976)}]{miura1976korteweg}
Miura, R.M. (1976).
\newblock The {K}orteweg--de{V}ries equation: {A} survey of results.
\newblock \emph{SIAM Review}, 18(3), 412--459.

\bibitem[{Okawa and Nonaka(2021)}]{okawa2021linear}
Okawa, I. and Nonaka, K. (2021).
\newblock Linear complementarity model predictive control with limited iterations for box-constrained problems.
\newblock \emph{Automatica}, 125, 109429.

\bibitem[{Proctor et~al.(2018)Proctor, Brunton, and Kutz}]{proctor2018generalizing}
Proctor, J., Brunton, S., and Kutz, J. (2018).
\newblock Generalizing {K}oopman theory to allow for inputs and control.
\newblock \emph{SIAM Journal on Applied Dynamical Systems}, 17(1), 909--930.

\bibitem[{Qin and Badgwell(2003)}]{qin2003survey}
Qin, S.J. and Badgwell, T.A. (2003).
\newblock A survey of industrial model predictive control technology.
\newblock \emph{Control Engineering Practice}, 11(7), 733--764.

\bibitem[{Richter et~al.(2011)Richter, Jones, and Morari}]{richter2011computational}
Richter, S., Jones, C.N., and Morari, M. (2011).
\newblock Computational complexity certification for real-time {MPC} with input constraints based on the fast gradient method.
\newblock \emph{IEEE Transactions on Automatic Control}, 57(6), 1391--1403.

\bibitem[{Stellato et~al.(2020)Stellato, Banjac, Goulart, Bemporad, and Boyd}]{stellato2020osqp}
Stellato, B., Banjac, G., Goulart, P., Bemporad, A., and Boyd, S. (2020).
\newblock {OSQP}: {A}n operator splitting solver for quadratic programs.
\newblock \emph{Mathematical Programming Computation}, 12(4), 637--672.

\bibitem[{Williams et~al.(2016)Williams, Hemati, Dawson, Kevrekidis, and Rowley}]{williams2016extending}
Williams, M., Hemati, M., Dawson, S., Kevrekidis, I., and Rowley, C. (2016).
\newblock Extending data-driven {K}oopman analysis to actuated systems.
\newblock \emph{IFAC-PapersOnLine}, 49(18), 704--709.

\bibitem[{Williams et~al.(2015)Williams, Kevrekidis, and Rowley}]{williams2015data}
Williams, M., Kevrekidis, I., and Rowley, C. (2015).
\newblock A data-driven approximation of the {K}oopman operator: {E}xtending dynamic mode decomposition.
\newblock \emph{Journal of Nonlinear Science}, 25, 1307--1346.

\bibitem[{Wu and Bemporad(2023{\natexlab{a}})}]{wu2023construction}
Wu, L. and Bemporad, A. (2023{\natexlab{a}}).
\newblock A construction-free coordinate-descent augmented-{L}agrangian method for embedded linear {MPC} based on {ARX} models.
\newblock \emph{IFAC-PapersOnLine}, 56(2), 9423--9428.

\bibitem[{Wu and Bemporad(2023{\natexlab{b}})}]{wu2023simple}
Wu, L. and Bemporad, A. (2023{\natexlab{b}}).
\newblock A {S}imple and {F}ast {C}oordinate-{D}escent {A}augmented-{L}agrangian {S}olver for {M}odel {P}redictive {C}ontrol.
\newblock \emph{IEEE Transactions on Automatic Control}, 68(11), 6860--6866.
\newblock \doi{10.1109/TAC.2023.3241238}.

\bibitem[{Wu and Braatz(2023)}]{wu2023direct}
Wu, L. and Braatz, R.D. (2023).
\newblock A direct optimization algorithm for input-constrained {MPC}.
\newblock \emph{arXiv preprint arXiv:2306.15079}.

\bibitem[{Zavala and Biegler(2009)}]{zavala2009advanced}
Zavala, V.M. and Biegler, L.T. (2009).
\newblock The advanced-step {NMPC} controller: {O}ptimality, stability and robustness.
\newblock \emph{Automatica}, 45(1), 86--93.

\end{thebibliography}

\end{document}